\documentclass{article}
\usepackage{amsfonts}
\usepackage{amsmath}
\usepackage{harvard}

\newtheorem{theorem}{Theorem}

\newtheorem{corollary}[theorem]{Corollary}

\newtheorem{definition}[theorem]{Definition}

\newtheorem{remark}[theorem]{Remark}

\newenvironment{proof}[1][Proof]{\noindent\textbf{#1.} }{\ \rule{0.5em}{0.5em}}
\numberwithin{theorem}{section}
\numberwithin{equation}{section}
\input{tcilatex}

\begin{document}

\title{Geometrical Objects on the First Order Jet Space $J^{1}(T,\mathbb{R}%
^{5})$ Produced by the Lorenz Atmospheric DEs System}
\author{Mircea Neagu \\
{\small 22 July 2007; Revised 7 February 2010 }\\
{\small (minor corrections; an added \textbf{Acknowledgements})}}
\date{}
\maketitle

\begin{abstract}
The aim of this paper is to construct natural geometrical objects on the
1-jet space $J^{1}(T,\mathbb{R}^{5})$, where $T\subset \mathbb{R}$, like a
non-linear connection, a generalized Cartan connection, together with its
d-torsions and d-curvatures, a jet electromagnetic d-field and a jet
Yang-Mills energy, starting from the given Lorenz atmospheric DEs system and
the pair of Euclidian metrics $\Delta =(1,\delta _{ij})$ on $T\times \mathbb{%
R}^{5}$.
\end{abstract}



\section{Jet Riemann-Lagrange geometry produced by a first order non-linear
DEs system}

\hspace{4mm}Many authors, like Asanov [1], Saunders [8] and many others,
studied the contravariant differential geometry of 1-jet spaces. Going on
with the geometrical studies of Asanov and using as a pattern the Lagrangian
geometrical ideas developed by Miron and Anastasiei [4], the author of this
paper has developed the \textit{Riemann-Lagrange geometry of 1-jet spaces}
[5], which is very suitable for the geometrical study of the solutions of a
given DEs or PDEs system, via the \textit{least squares variational method}
proposed by Udri\c{s}te and Neagu in [7], [9].

In what follows we present the main jet Riemann-Lagrange geometrical results
that, in author opinion, characterize a given non-linear DEs system of order
one. In this direction, let $T=[a,b]\subset \mathbb{R}$ be a compact
interval of the set of real numbers and let us consider the jet fibre bundle
of order one%
\begin{equation*}
J^{1}(T,\mathbb{R}^{n})\rightarrow T\times \mathbb{R}^{n},\text{ }n\geq 2,
\end{equation*}%
whose local coordinates $(t,x^{i},x_{1}^{i}),$ $i=\overline{1,n},$ transform
by the rules%
\begin{equation*}
\widetilde{t}=\widetilde{t}(t),\text{ }\widetilde{x}^{i}=\widetilde{x}%
^{i}(x^{j}),\text{ }\widetilde{x}_{1}^{i}=\frac{\partial \widetilde{x}^{i}}{%
\partial x^{j}}\frac{dt}{d\widetilde{t}}\cdot x_{1}^{j}.
\end{equation*}

\begin{remark}
From a physical point of view, the coordinate $t$ has the physical meaning
of \textbf{relativistic time}, the coordinates $(x^{i})_{i=\overline{1,n}}$
represent \textbf{spatial coordinates} and the coordinates $(x_{1}^{i})_{i=%
\overline{1,n}}$ have the physical meaning of \textbf{relativistic velocities%
}.
\end{remark}

Let us consider that $X=\left( X_{(1)}^{(i)}(x^{k})\right) $ is an arbitrary
d-tensor field on the 1-jet space $J^{1}(T,\mathbb{R}^{n})$, whose local
components transform by the rules%
\begin{equation*}
\widetilde{X}_{(1)}^{(i)}=\frac{\partial \widetilde{x}^{i}}{\partial x^{j}}%
\frac{dt}{d\widetilde{t}}\cdot X_{(1)}^{(j)}.
\end{equation*}

It is obvious that the d-tensor field $X$ produces the jet first order DEs
system (\textit{jet dynamical system})%
\begin{equation}
x_{1}^{i}=X_{(1)}^{(i)}(x^{k}(t)),\text{ }\forall \text{ }i=\overline{1,n},
\label{DEs1}
\end{equation}%
where $c(t)=(x^{i}(t))$ is an unknown curve on $\mathbb{R}^{n}$ (i. e., a
jet field line of the d-tensor field $X$) and we used the notation%
\begin{equation*}
x_{1}^{i}\overset{not}{=}\dot{x}^{i}=\frac{dx^{i}}{dt},\text{ }\forall \text{
}i=\overline{1,n}.
\end{equation*}

Automatically, the jet first order DEs system (\ref{DEs1}), together with
the pair of Euclidian metrics $\Delta =(1,\delta _{ij})$ on $T\times \mathbb{%
R}^{n}$, produces the \textit{jet least squares Lagrangian function} 
\begin{equation*}
JLS_{\Delta }^{\text{DEs}}:J^{1}(T,\mathbb{R}^{n})\rightarrow \mathbb{R}_{+},
\end{equation*}%
expressed by%
\begin{eqnarray}
JLS_{\Delta }^{\text{DEs}}(x^{k},x_{1}^{k}) &=&\sum_{i,j=1}^{n}\delta _{ij}%
\left[ x_{1}^{i}-X_{(1)}^{(i)}(x)\right] \left[ x_{1}^{j}-X_{(1)}^{(j)}(x)%
\right] =  \notag \\
&=&\sum_{i=1}^{n}\left[ x_{1}^{i}-X_{(1)}^{(i)}(x)\right] ^{2},
\label{JetLS}
\end{eqnarray}%
where $x=(x^{k})_{k=\overline{1,n}}.$ Because the \textit{global minimum
points} of the \textit{jet least squares energy action}%
\begin{equation*}
\mathbb{E}_{\Delta }^{\text{DEs}}(c(t))=\int_{a}^{b}JLS_{\Delta }^{\text{DEs}%
}(x^{k}(t),\dot{x}^{k}(t))dt
\end{equation*}%
are exactly the solutions of class $C^{2}$ of the jet first order DEs system
(\ref{DEs1}), it follows that we may regard the jet least squares Lagrangian
function $JLS_{\Delta }^{\text{DEs}}$ as a natural geometrical substitut for
the DEs system of order one (\ref{DEs1}), on the 1-jet space $J^{1}(T,%
\mathbb{R}^{n})$.

\begin{remark}
It is important to note that any solution of class $C^{2}$ of the jet first
order DEs system (\ref{DEs1}) verify the second order Euler-Lagrange
equations produced by the jet least squares Lagrangian function $JLS_{\Delta
}^{\text{DEs}}$ (\textit{\ }\textbf{jet geometric dynamics})%
\begin{equation}
\frac{\partial \left[ JLS_{\Delta }^{\text{DEs}}\right] }{\partial x^{i}}-%
\frac{d}{dt}\left( \frac{\partial \left[ JLS_{\Delta }^{\text{DEs}}\right] }{%
\partial \dot{x}^{i}}\right) =0,\text{ }\forall \text{ }i=\overline{1,n}.
\label{E-L}
\end{equation}

Conversely, this statement is not true because there exist solutions for the
Euler-Lagrange DEs system (\ref{E-L}) which are not global minimum points
for the jet least squares energy action $\mathbb{E}_{\Delta }^{\text{DEs}}$,
that is which are not solutions for the jet first order DEs system (\ref%
{DEs1}).
\end{remark}

But, a Riemann-Lagrange geometry on $J^{1}(T,\mathbb{R}^{n})$ produced by
the jet least squares Lagrangian function $JLS_{\Delta }^{\text{DEs}}$, via
its Euler-Lagrange equations (\ref{E-L}), geometry in the sense of
non-linear connection, generalized Cartan connection, d-torsions,
d-curvatures, jet electromagnetic field and jet Yang-Mills energy, is now
completely done in the papers [5], [6] and [7]. For that reason, we
introduce the following concept:

\begin{definition}
Any geometrical object on $J^{1}(T,\mathbb{R}^{n})$, which is produced by
the jet least squares Lagrangian function $JLS_{\Delta }^{\text{DEs}}$, via
its second order Euler-Lagrange equations (\ref{E-L}), is called \textbf{%
geometrical object produced by the jet first order DEs system (\ref{DEs1})}.
\end{definition}

In a such context, we give the following geometrical result, which is proved
in [6] and, for the multi-time general case, in [7]. For all details, the
reader is invited to consult the book [5].

\begin{theorem}
\label{MainTh}(i) The \textbf{canonical non-linear connection on }$J^{1}(T,%
\mathbb{R}^{n})$\textbf{\ produced by the jet first order DEs system (\ref%
{DEs1})} has the local components%
\begin{equation*}
\Gamma ^{\text{DEs}}=\left( M_{(1)1}^{(i)},N_{(1)j}^{(i)}\right) ,
\end{equation*}%
where%
\begin{equation*}
M_{(1)1}^{(i)}=0\text{ and }N_{(1)j}^{(i)}=-\frac{1}{2}\left[ \frac{\partial
X_{(1)}^{(i)}}{\partial x^{j}}-\frac{\partial X_{(1)}^{(j)}}{\partial x^{i}}%
\right] ,\text{ }\forall \text{ }i,j=\overline{1,n}.
\end{equation*}

(ii) All adapted components of the \textbf{canonical generalized Cartan
connection }$C\Gamma ^{\text{DEs}}$\textbf{\ produced by the jet first order
DEs system (\ref{DEs1})} vanish.

(iii) The effective adapted components $R_{(1)jk}^{(i)}$ of the \textbf{%
torsion} d-tensor \textbf{T}$^{\text{DEs}}$ of the canonical generalized
Cartan connection $C\Gamma ^{\text{DEs}}$ \textbf{produced by the jet first
order DEs system (\ref{DEs1})} are%
\begin{equation*}
R_{(1)jk}^{(i)}=-\frac{1}{2}\left[ \frac{\partial ^{2}X_{(1)}^{(i)}}{%
\partial x^{k}\partial x^{j}}-\frac{\partial ^{2}X_{(1)}^{(j)}}{\partial
x^{k}\partial x^{i}}\right] ,\text{ }\forall \text{ }i,j,k=\overline{1,n}.
\end{equation*}

(iv) All adapted components of the \textbf{curvature} d-tensor \textbf{R}$^{%
\text{DEs}}$ of the canonical generalized Cartan connection $C\Gamma ^{\text{%
DEs}}$ \textbf{produced by the jet first order DEs system (\ref{DEs1})}
vanish.

(v) The \textbf{geometric electromagnetic distinguished 2-form produced by
the jet first order DEs system (\ref{DEs1})} has the expression%
\begin{equation*}
F^{\text{DEs}}=F_{(i)j}^{(1)}\delta x_{1}^{i}\wedge dx^{j},
\end{equation*}%
where%
\begin{equation*}
\delta x_{1}^{i}=dx_{1}^{i}+N_{(1)k}^{(i)}dx^{k},\text{ }\forall \text{ }i=%
\overline{1,n},
\end{equation*}%
and%
\begin{equation*}
F_{(i)j}^{(1)}=\frac{1}{2}\left[ \frac{\partial X_{(1)}^{(i)}}{\partial x^{j}%
}-\frac{\partial X_{(1)}^{(j)}}{\partial x^{i}}\right] ,\text{ }\forall 
\text{ }i,j=\overline{1,n}.
\end{equation*}

(vi) The adapted components $F_{(i)j}^{(1)}$ of the electromagnetic d-form $%
F^{\text{DEs}}$ produced by the jet first order DEs system (\ref{DEs1})%
\textbf{\ }verify the \textbf{generalized Maxwell equations}%
\begin{equation*}
\sum_{\{i,j,k\}}F_{(i)j||k}^{(1)}=0,
\end{equation*}%
where $\sum_{\{i,j,k\}}$ represents a cyclic sum and%
\begin{equation*}
F_{(i)j||k}^{(1)}=\frac{\partial F_{(i)j}^{(1)}}{\partial x^{k}}
\end{equation*}%
has the geometrical meaning of the horizontal local covariant derivative
produced by the Berwald linear connection $B\Gamma _{0}$ on $J^{1}(T,\mathbb{%
R}^{n}).$ For more details, please consult [5].

(vii) The \textbf{geometric jet Yang-Mills energy produced by the jet first
order DEs system (\ref{DEs1})} is defined by the formula%
\begin{equation*}
EYM^{\text{DEs}}(x)=\sum_{i=1}^{n-1}\sum_{j=i+1}^{n}\left[ F_{(i)j}^{(1)}%
\right] ^{2}.
\end{equation*}
\end{theorem}

\begin{remark}
\label{F} If we use the following matriceal notations

\begin{itemize}
\item $J\left( X_{(1)}\right) =\left( \dfrac{\partial X_{(1)}^{(i)}}{%
\partial x^{j}}\right) _{i,j=\overline{1,n}}$ - the \textbf{Jacobian matrix},

\item $N_{(1)}=\left( N_{(1)j}^{(i)}\right) _{i,j=\overline{1,n}}$ - the 
\textbf{non-linear connection matrix},

\item $R_{(1)k}=\left( R_{(1)jk}^{(i)}\right) _{i,j=\overline{1,n}},$ $%
\forall $ $k=\overline{1,n},$ - the \textbf{torsion matrices},

\item $F^{(1)}=\left( F_{(i)j}^{(1)}\right) _{i,j=\overline{1,n}}$ - the 
\textbf{electromagnetic matrix},
\end{itemize}

then the following matriceal geometrical relations attached to the jet first
order DEs system (\ref{DEs1}) hold good:

\begin{enumerate}
\item $N_{(1)}=-\dfrac{1}{2}\left[ J\left( X_{(1)}\right) -\text{ }%
^{T}J\left( X_{(1)}\right) \right] ,$

\item $R_{(1)k}=\dfrac{\partial }{\partial x^{k}}\left[ N_{(1)}\right] ,$ $%
\forall $ $k=\overline{1,n},$

\item $F^{(1)}=-N_{(1)},$

\item $EYM^{\text{DEs}}(x)=\dfrac{1}{2}\cdot Trace\left[ F^{(1)}\cdot \text{ 
}^{T}F^{(1)}\right] ,$ that is the jet Yang-Mills energy coincides with the
norm of the skew-symmetric electromagnetic matrix $F^{(1)}$ in the Lie
algebra $o(n)=L(O(n)).$
\end{enumerate}
\end{remark}

In the sequel, we apply the above jet contravariant Riemann-Lagrange
geometrical results to the Lorenz five-components atmospheric DEs system
introduced by Lorenz [3] and studied, via the Melnikov function method for
Hamiltonian systems on Lie groups, by Birtea, Puta, Ra\c{t}iu and Tudoran
[2].

\section{Jet Riemann-Lagrange geometry produced by the Lorenz simplified
model of Rossby gravity wave interaction}

\hspace{4mm} The first model equations for the atmosphere are that so called 
\textit{primitive equations}. It seems that this model produces wave-like
motions on different time scales:

\begin{itemize}
\item on the one hand, this model produces the slow motions which have a
period of order of days (these slow-waves are called \textit{Rossby waves});

\item on the other hand, this model produces fast motions which have a
period of hours (these fast-waves are called \textit{gravity waves}).
\end{itemize}

The question of how to balance these two time scales leads Lorenz [3] to
consider a simplified version of the primitive equations model, which is
given by the following non-linear system of five differential equations [2]:%
\begin{equation}
\left\{ 
\begin{array}{l}
\dfrac{dx^{1}}{dt}=-x^{2}x^{3}+\varepsilon x^{2}x^{5}\medskip \\ 
\dfrac{dx^{2}}{dt}=x^{1}x^{3}-\varepsilon x^{1}x^{5}\medskip \\ 
\dfrac{dx^{3}}{dt}=-x^{1}x^{2}\medskip \\ 
\dfrac{dx^{4}}{dt}=-x^{5}\medskip \\ 
\dfrac{dx^{5}}{dt}=x^{4}+\varepsilon x^{1}x^{2},%
\end{array}%
\right.  \label{Lorenz}
\end{equation}%
where the variables $x^{4}$, $x^{5}$ represent the fast gravity wave
oscillations and the variables $x^{1}$, $x^{2}$, $x^{3}$ are the slow Rossby
wave oscillations, with a parameter $\varepsilon $ which is related to the
physical Rossby number.

\begin{remark}
It is obvious that, from a physical point of view, the Lorenz atmospheric
DEs system (\ref{Lorenz}) couples the Rossby waves with the gravity waves.
\end{remark}

Naturally, the Lorenz atmospheric DEs system (\ref{Lorenz}) can be regarded
as a non-linear DEs system of order one on the 1-jet space $J^{1}(T,\mathbb{R%
}^{5})$, which is produced by the d-tensor field $X=\left(
X_{(1)}^{(i)}(x)\right) $, where $i=\overline{1,5}$ and%
\begin{equation*}
x=(x^{1},x^{2},x^{3},x^{4},x^{5}),
\end{equation*}%
having the local components%
\begin{equation}
\begin{array}{l}
X_{(1)}^{(1)}(x)=-x^{2}x^{3}+\varepsilon x^{2}x^{5},\medskip \\ 
X_{(1)}^{(2)}(x)=x^{1}x^{3}-\varepsilon x^{1}x^{5},\medskip \\ 
X_{(1)}^{(3)}(x)=-x^{1}x^{2},\medskip \\ 
X_{(1)}^{(4)}(x)=-x^{5},\medskip \\ 
X_{(1)}^{(5)}(x)=x^{4}+\varepsilon x^{1}x^{2}.%
\end{array}
\label{XLorenz}
\end{equation}

Consequently, via the Theorem \ref{MainTh}, we assert that the
Riemann-Lagrange geometrical behavior on the 1-jet space $J^{1}(T,\mathbb{R}%
^{5})$ of the Lorenz atmospheric DEs system (\ref{Lorenz}) is described in
the following

\begin{corollary}
(i) The \textbf{canonical non-linear connection on }$J^{1}(T,\mathbb{R}^{5})$%
\textbf{\ produced by the Lorenz atmospheric DEs system (\ref{Lorenz})} has
the local components%
\begin{equation*}
\hat{\Gamma}=\left( 0,\hat{N}_{(1)j}^{(i)}\right) ,
\end{equation*}%
where $\hat{N}_{(1)j}^{(i)}$ are the entries of the matrix%
\begin{equation*}
\hat{N}_{(1)}=\left( \hat{N}_{(1)j}^{(i)}\right) _{i,j=\overline{1,5}%
}=\left( 
\begin{array}{ccccc}
0 & x^{3}-\varepsilon x^{5} & 0 & 0 & 0\medskip \\ 
-x^{3}+\varepsilon x^{5} & 0 & -x^{1} & 0 & \varepsilon x^{1}\medskip \\ 
0 & x^{1} & 0 & 0 & 0\medskip \\ 
0 & 0 & 0 & 0 & 1\medskip \\ 
0 & -\varepsilon x^{1} & 0 & -1 & 0%
\end{array}%
\right) .
\end{equation*}

(ii) All adapted components of the \textbf{canonical generalized Cartan
connection }$C\hat{\Gamma}$\textbf{\ produced by the Lorenz atmospheric DEs
system (\ref{Lorenz})} vanish.

(iii) All adapted components of the \textbf{torsion} d-tensor \textbf{\^{T}}
of the canonical generalized Cartan connection $C\hat{\Gamma}$ \textbf{%
produced by the Lorenz atmospheric DEs system (\ref{Lorenz})} are zero,
except%
\begin{equation*}
\begin{array}{ll}
\hat{R}_{(1)21}^{(3)}=-\hat{R}_{(1)31}^{(2)}=1, & \hat{R}_{(1)21}^{(5)}=-%
\hat{R}_{(1)51}^{(2)}=-\varepsilon ,\medskip \\ 
\hat{R}_{(1)13}^{(2)}=-\hat{R}_{(1)23}^{(1)}=-1, & \hat{R}_{(1)15}^{(2)}=-%
\hat{R}_{(1)25}^{(1)}=\varepsilon .%
\end{array}%
\end{equation*}

(iv) All adapted components of the \textbf{curvature} d-tensor \textbf{\^{R}}
of the canonical generalized Cartan connection $C\hat{\Gamma}$ \textbf{%
produced by the Lorenz atmospheric DEs system (\ref{Lorenz})} vanish.

(v) The \textbf{geometric electromagnetic distinguished 2-form produced by
the Lorenz atmospheric DEs system (\ref{Lorenz})} has the expression%
\begin{equation*}
\hat{F}=\hat{F}_{(i)j}^{(1)}\delta x_{1}^{i}\wedge dx^{j},
\end{equation*}%
where%
\begin{equation*}
\delta x_{1}^{i}=dx_{1}^{i}+\hat{N}_{(1)k}^{(i)}dx^{k},\text{ }\forall \text{
}i=\overline{1,5},
\end{equation*}%
and the adapted components $\hat{F}_{(i)j}^{(1)}$ are the entries of the
matrix%
\begin{equation*}
\hat{F}^{(1)}=\left( \hat{F}_{(i)j}^{(1)}\right) _{i,j=\overline{1,5}}=-\hat{%
N}_{(1)}.
\end{equation*}

(vi) The \textbf{jet geometric Yang-Mills energy produced by the Lorenz
atmospheric DEs system (\ref{Lorenz})} is given by the formula%
\begin{equation*}
EYM^{\text{Lorenz}}(x)=\left( \varepsilon x^{5}-x^{3}\right) ^{2}+\left(
x^{1}\right) ^{2}+\left( \varepsilon x^{1}\right) ^{2}+1.
\end{equation*}
\end{corollary}

\begin{proof}
The Lorenz atmospheric DEs system (\ref{Lorenz}) is a particular case of the
jet first order DEs system (\ref{DEs1}) for $n=5$ and $X=\left(
X_{(1)}^{(i)}(x)\right) _{i=\overline{1,5}}$ given by the relations (\ref%
{XLorenz}). In conclusion, applying the Theorem \ref{MainTh}, together with
the Remark \ref{F}, and using the Jacobian matrix%
\begin{equation*}
J\left( X_{(1)}\right) =\left( 
\begin{array}{ccccc}
0 & -x^{3}+\varepsilon x^{5} & -x^{2} & 0 & \varepsilon x^{2}\medskip \\ 
x^{3}-\varepsilon x^{5} & 0 & x^{1} & 0 & -\varepsilon x^{1}\medskip \\ 
-x^{2} & -x^{1} & 0 & 0 & 0\medskip \\ 
0 & 0 & 0 & 0 & -1\medskip \\ 
\varepsilon x^{2} & \varepsilon x^{1} & 0 & 1 & 0%
\end{array}%
\right) ,
\end{equation*}%
we obtain what we were looking for.
\end{proof}

\begin{remark}
Let us remark that, although the jet Yang-Mills electromagnetic energy $EYM^{%
\text{Lorenz}}$ produced by the Lorenz atmospheric DEs system (\ref{Lorenz})
depends only by the coordinates $x^{1}$, $x^{3}$ and $x^{5}$, it still
couples the slow Rossby wave oscillations with the fast gravity wave
oscillations. However, the coordinates $x^{2}$ and $x^{4}$ are missing in
the expression of $EYM^{\text{Lorenz}}$. There exists a physical
interpretation of this fact?
\end{remark}

\section{Yang-Mills energetical hypersurfaces of con\-stant level produced
by the Lorenz at\-mos\-phe\-ric DEs system}

\hspace{4mm} In the preceding Riemann-Lagrange geometrical theory on the
1-jet space $J^{1}(T,\mathbb{R}^{5})$ the Lorenz atmospheric DEs system (\ref%
{Lorenz}) \textit{"produces"} a jet Yang-Mills energy given by the formula%
\begin{equation*}
EYM^{\text{Lorenz}}(x)=\left( 1+\varepsilon ^{2}\right) \left( x^{1}\right)
^{2}+\left( x^{3}\right) ^{2}+\varepsilon ^{2}(x^{5})^{2}-2\varepsilon
x^{3}x^{5}+1,
\end{equation*}%
where $x=(x^{1},x^{2},x^{3},x^{4},x^{5}).$ In what follows, let us study the 
\textit{jet Yang-Mills energetical hypersurfaces of constant level} produced
by the Lorenz at\-mos\-phe\-ric DEs system (\ref{Lorenz}), which are defined
by the implicit equations%
\begin{equation*}
\Sigma _{C}^{\text{Lorenz}}:\left( \varepsilon x^{5}-x^{3}\right)
^{2}+\left( 1+\varepsilon ^{2}\right) \left( x^{1}\right) ^{2}=C-1,
\end{equation*}%
where $C$ is a constant real number.

Because $\Sigma _{C}^{\text{Lorenz}}$ is a \textit{quadric} in the system of
axes $Ox^{1}x^{3}x^{5}$ for every $C\in \mathbb{R}$, then, using the
reduction to the canonical form of a quadric, we find the following
geometrical results:

\begin{enumerate}
\item If $C<1$, then we have $\Sigma _{C<1}^{\text{Lorenz}}=\emptyset $;

\item If $C=1$, then we have%
\begin{equation*}
\Sigma _{C=1}^{\text{Lorenz}}:\left\{ 
\begin{array}{l}
x^{1}=0\medskip \\ 
x^{3}-\varepsilon x^{5}=0,%
\end{array}%
\right.
\end{equation*}%
that is $\Sigma _{C=1}^{\text{Lorenz}}$ is a \textit{straight line} in the
system of axes $Ox^{1}x^{3}x^{5}$;

\item If $C>1$, then we have%
\begin{equation*}
\Sigma _{C>1}^{\text{Lorenz}}:\left( 1+\varepsilon ^{2}\right) \left(
x^{1}\right) ^{2}+\left( x^{3}\right) ^{2}+\varepsilon
^{2}(x^{5})^{2}-2\varepsilon x^{3}x^{5}-C+1=0,
\end{equation*}%
that is $\Sigma _{C>1}^{\text{Lorenz}}$ is a degenerate non-empty quadric in
the system of axes $Ox^{1}x^{3}x^{5}$, whose canonical form is%
\begin{equation*}
\Sigma _{C>1}^{\text{Lorenz}}:\left( X^{3}\right) ^{2}+\left( X^{5}\right)
^{2}=\frac{C-1}{1+\varepsilon ^{2}},
\end{equation*}%
where the rotation of the system of axes $Ox^{1}x^{3}x^{5}$ into the system
of axes $OX^{1}X^{3}X^{5}$ is given by the matriceal relation%
\begin{equation*}
\left( 
\begin{array}{c}
x^{1} \\ 
x^{3} \\ 
x^{5}%
\end{array}%
\right) =\frac{1}{\sqrt{1+\varepsilon ^{2}}}\left( 
\begin{array}{ccc}
0 & \sqrt{1+\varepsilon ^{2}} & 0 \\ 
\varepsilon & 0 & 1 \\ 
1 & 0 & -\varepsilon%
\end{array}%
\right) \left( 
\begin{array}{c}
X^{1} \\ 
X^{3} \\ 
X^{5}%
\end{array}%
\right) .
\end{equation*}

\hspace{4mm} In conclusion, the degenerate non-empty quadric $\Sigma _{C>1}^{%
\text{Lorenz}}$ is in the system of axes $Ox^{1}x^{3}x^{5}$ a \textit{slant
circular cylinder} of radius%
\begin{equation*}
R=\sqrt{\frac{C-1}{1+\varepsilon ^{2}}},
\end{equation*}%
having as axis of symmetry the straight line $\Sigma _{C=1}^{\text{Lorenz}}$.
\end{enumerate}

\textbf{Open problem. }There exist real physical interpretations, in the
study of the Lorenz atmospheric DEs system (\ref{Lorenz}), for the preceding
geometrical results?

\textbf{Acknowledgements.} This work was supported by Research Contract with
Sinoptix No. 8441/2009. The author would like to express his sincere
gratitude to Professor Gh. Atanasiu for its suggestions and useful
discussions on this topic.

\textbf{Author's address:} Mircea N{\scriptsize EAGU}

University Transilvania of Brasov, Faculty of Mathematics and Informatics

Department of Algebra, Geometry and Differential Equations

B-dul Iuliu Maniu, No. 50, BV 500091, Brasov, Romania

\textbf{E-mails:} mircea.neagu@unitbv.ro, mirceaneagu73@yahoo.com

\textbf{Website:} http://www.2collab.com/user:mirceaneagu

\end{document}